\documentclass[11pt]{amsart}

\usepackage{enumerate,url,amssymb,  mathrsfs, graphicx, pdfsync}
\newtheorem{theorem}{Theorem}[section]
\newtheorem{lemma}[theorem]{Lemma}
\newtheorem*{lemma*}{Lemma}

\newtheorem{proposition}[theorem]{Proposition}

\theoremstyle{definition}
\newtheorem{definition}[theorem]{Definition}
\newtheorem{example}[theorem]{Example}

\newtheorem{question}[theorem]{Question}

\theoremstyle{remark}

\numberwithin{equation}{section}

\renewcommand{\bar}[1]{\overline{#1}}

\newcommand{\xx}{\mathbb{X}}
\newcommand{\yy}{\mathbb{Y}}
\newcommand{\cc}{\mathbb{C}}
\newcommand{\dd}{\mathbb{D}}

\newcommand{\abs}[1]{\lvert#1\rvert}

\newcommand{\A}{\mathbb{A}}
\newcommand{\C}{\mathbb{C}}

\newcommand{\W}{\mathscr{W}}

\newcommand{\X}{\mathbb{X}}

\newcommand{\Y}{\mathbb{Y}}

\newcommand{\dtext}{\textnormal d}

\newcommand{\onto}{\xrightarrow[]{{}_{\!\!\textnormal{onto\,\,}\!\!}}}

\DeclareMathOperator{\dist}{dist}

\DeclareMathOperator{\re}{Re}
\DeclareMathOperator{\im}{Im}

\DeclareMathOperator{\loc}{loc}

\DeclareMathOperator{\Div}{div}

\DeclareMathOperator{\osc}{osc}

\def\leq{\leqslant}
\def\geq{\geqslant}

\def\le{\leqslant}
\def\ge{\geqslant}

\def\XXint#1#2#3{{\setbox0=\hbox{$#1{#2#3}{\int}$}\vcenter{\hbox{$#2#3$}}\kern-.5\wd0}}

\def\XXiint#1#2#3{{\setbox0=\hbox{$#1{#2#3}{\iint}$}\vcenter{\hbox{$#2#3$}}\kern-.5\wd0}}

\begin{document}
\title{Sobolev homeomorphic extensions}


\author[A. Koski]{Aleksis Koski}
\address{Department of Mathematics and Statistics, P.O.Box 35 (MaD) FI-40014 University of Jyv\"askyl\"a, Finland}
\email{aleksis.t.koski@jyu.fi}

\author[J. Onninen]{Jani Onninen}
\address{Department of Mathematics, Syracuse University, Syracuse,
NY 13244, USA and  Department of Mathematics and Statistics, P.O.Box 35 (MaD) FI-40014 University of Jyv\"askyl\"a, Finland
}
\email{jkonnine@syr.edu}
\thanks{  A. Koski was supported by the Academy of Finland Grant number 307023.
J. Onninen was supported by the NSF grant  DMS-1700274.}




\begin{abstract} 
Let $\X$ and $\Y$ be $\ell$-connected Jordan domains, $\ell \in \mathbb N$, with rectifiable boundaries in the complex plane. We prove that any boundary homeomorphism $\varphi \colon \partial \X \onto \partial \Y$ admits a Sobolev homeomorphic extension $h \colon \overline{\X} \onto \overline{\Y}$ in $\W^{1,1} (\X, \C)$. If instead $\X$ has $s$-hyperbolic growth with $s>p-1$, we show the existence of such an  extension lies in the Sobolev class $\W^{1,p} (\X, \mathbb C)$ for $p\in (1,2)$. Our examples show that the assumptions of rectifiable boundary and hyperbolic growth  cannot be relaxed. We also consider the existence of $\W^{1,2}$-homeomorphic extensions subject to a given boundary data.
\end{abstract}

\maketitle
\section{Introduction}

Throughout this text $\X$ and $\Y$ are $\ell$-connected Jordan domains, $\ell=1,2, \dots$, in the complex plane $\mathbb C$. Their boundaries $\partial\X$ and $\partial\Y$ are thus a disjoint union of $\ell$ simple closed curves. If $\ell =1 $, these domains are simply connected and will just be called Jordan domains. In the simply connected case, the Jordan-Sch\"onflies theorem states that every homeomorphism $\varphi \colon  \partial \X \onto \partial \Y$ admits a continuous extension $h \colon \overline{\X }\to  \overline{\Y}$ which takes $\X$ homeomorphically onto $\Y$.   In the first part of this  paper we focus on a Sobolev variant of the Jordan-Sch\"onflies theorem. The most pressing demand for studying such  variants comes from the variational approach to Geometric Function Theory~\cite{AIMb, IMb, Reb} and  Nonlinear Elasticity~\cite{Anb, Bac, Cib}. Both theories share the compilation ideas to determine the infimum of  a given energy functional 
\begin{equation}\label{energ}
\mathsf E_\mathbb X [h] =  \int_\mathbb X \mathbf {\bf E}(x,h, Dh )\,  \dtext x\, , 
\end{equation}
among  orientation preserving homeomorphisms $h \colon \overline{\mathbb X} \onto \overline{\mathbb Y }$  in the  Sobolev space $\mathscr W^{1,p} (\X, \Y)$ with given boundary data $\varphi \colon \partial \X \onto \partial \Y$.  We denote such a class of mappings by $\mathscr H_\varphi^{1,p}(\mathbb X, \mathbb Y)$.   Naturally, a fundamental question to raise then is whether the class  $\mathscr H_\varphi^{1,p}(\mathbb X, \mathbb Y)$ is non-empty.

\begin{question}\label{existQuestion}  Under what conditions does a given boundary homeomorphism $\varphi \colon \partial \X \onto \partial \Y$ admit a homeomorphic extension $h \colon \overline {\X} \onto \overline{\Y}$
of Sobolev class $\W^{1,p} (\X, \C)$?
\end{question}

A necessary condition is that the mapping $\varphi$ is the Sobolev trace of some (possibly non-homeomorphic) mapping in $\W^{1,p}(\xx, \C)$. Hence to solve Question \ref{existQuestion} one could first study the following natural sub-question:
\begin{question}\label{Q2}
Suppose that a homeomorphism $\varphi : \partial \xx \to \partial \yy$ admits a Sobolev $\W^{1,p}$-extension to $\xx$. Does it then follow that $\varphi$ also admits a homeomorphic Sobolev $\W^{1,p}$-extension to $\xx$?
\end{question}

 Our main results, Theorem~\ref{thm:main} and its multiply connected variant (Theorem~\ref{thm:multiply}), give an answer to these questions when $p \in [1,2)$. The construction of such extensions is important not only to ensure the well-posedness of the related variational questions, but also for example due to the fact that various types of extensions were used to provide approximation results for Sobolev homeomorphisms, see \cite{HP, IKOapprox}. We touch upon the variational topics in Section~\ref{sec:mono}, where we provide an application for one of our results.  Apart from Theorem~\ref{thm:multiply} and its proof (\S\ref{anyplansguysz}), the rest of the paper deals with the simply connected case.

Let us start considering  the above questions in the well-studied setting of the Dirichlet energy, corresponding to $p=2$ above. The  Rad\'o
\cite{Ra}, Kneser \cite{Kn} and Choquet \cite{Ch} theorem  asserts  that if ${\mathbb Y}
\subset {\mathbb R}^2$ is a  convex domain then the harmonic extension of a
homeomorphism $\varphi \colon 
\partial {\mathbb X} \to \partial {\mathbb Y}$  is a univalent map from ${\mathbb X}$
onto ${\mathbb Y}$. Moreover, by a theorem of Lewy \cite{Le},
this univalent harmonic map has a non-vanishing Jacobian and is therefore a real analytic diffeomorphism in $\xx$. However, such an extension is not guaranteed to have finite Dirichlet energy in $\xx$. The class of boundary functions which admit a harmonic extension with finite Dirichlet energy was  characterized by Douglas~\cite{Do}.  The {\it Douglas condition} for a function $\varphi \colon \partial \mathbb D \onto \partial \Y$ reads as
\begin{equation}\label{eq:douglas}
\int_{\partial \mathbb D} \int_{\partial \mathbb D} \left| \frac{\varphi (\xi) - \varphi (\eta)}{ \xi - \eta }\right|^2 \abs{\dtext \xi } \,  \abs{\dtext \eta }  < \infty \, .
\end{equation}
The mappings satisfying this condition are exactly the ones that admit an extension with finite $\W^{1,2}$-norm. Among these extensions is the harmonic extension of $\varphi$, which is known to have the smallest Dirichlet energy among all extensions.

Note that the Dirichlet energy is also invariant with respect to a conformal change of variables in the domain $\xx$. Therefore thanks to the Riemann Mapping Theorem, when considering Question~\ref{existQuestion} in the case $p=2$, we may assume that $\X = \mathbb D$  without loss of generality. Now, there is no challenge to answer Question~\ref{existQuestion} when $p=2$ and $\Y$  is Lipschitz. Indeed, for any Lipschitz domain there exists a global bi-Lipschitz change of  variables $\Phi \colon \mathbb C \to \mathbb C$ for which $\Phi (\Y)$ is the unit disk. Since the finiteness of the Dirichlet energy is preserved under a bi-Lipschitz change of variables in the target, we may reduce Question~\ref{existQuestion} to the case when $\xx = \yy = \dd$, for which the Rad\'o-Kneser-Choquet theorem and the Douglas condition provide an answer. In other words, if $\Y$ is Lipschitz then the following are equivalent for a boundary homeomorphism $\varphi \colon \partial \mathbb D \to \partial \Y$

\begin{enumerate}
\item{$\varphi$ admits a $\W^{1,2}$-Sobolev homeomorphic extension $h \colon \overline{\mathbb D} \onto \overline{\Y}$}\label{pt:1}
\item{$\varphi$ admits $\W^{1,2}$-Sobolev extension to  $\mathbb D$}\label{pt:2}
\item{$\varphi$ satisfies the Douglas condition~\eqref{eq:douglas}}\label{pt:3}
\end{enumerate}

In the case when $1 \leq p < 2$, the problem is not invariant under a conformal change of variables in $\X$. However,  when $\xx$ is the unit disk and $\yy$ is a convex domain, a complete answer to Question~\ref{existQuestion} was provided by the following result of Verchota~\cite{Ve}.
\begin{proposition}\label{Verchota} Let $\Y$ be a convex domain, and let $\varphi : \partial \dd \onto \partial \Y$ be any homeomorphism. Then the harmonic extension of $\varphi$ lies in the Sobolev class $\W^{1,p}(\dd,\cc)$ for all $1 \le p < 2$.
\end{proposition}

This result was further generalized in~\cite{IMS} and~\cite{Kal}. The case $p>2$ will be discussed in Section~\ref{sec:liptriv}. Our main purpose   is to provide a general study of Question~\ref{existQuestion} in the case when $1 \leq p < 2$.

Considering now the endpoint case $p = \infty$, we find that Question~\ref{existQuestion} is equivalent to the question of finding a homeomorphic Lipschitz map extending the given boundary data $\varphi$. In this case the Kirszbraun extension theorem~\cite{Ki} shows that a boundary map $\varphi \colon \partial \mathbb D \onto \partial \Y$ admits a Lipschitz extension if and only if $\varphi$ is a Lipschitz map itself. In the case when $\xx$ is the unit disk, a positive answer to Question~\ref{Q2} is then given by the following recent result by   Kovalev~\cite{Kovalev}.

\begin{theorem}\label{thm:lip}\textnormal{{\bf ($p=\infty$)}} Let $\varphi \colon \partial \mathbb D \to \cc$ be a Lipschitz embedding. Then $\varphi$ admits a homeomorphic Lipschitz extension to the whole plane $\cc$.
\end{theorem}

Let us return to the case of the Dirichlet energy, see~\eqref{pt:1}-\eqref{pt:3} above.  The equivalence of a $\W^{1,2}$-Sobolev extension and a $\W^{1,2}$-Sobolev homeomorphic extension for non-Lipschitz targets is a more subtle question. In this perspective,  a slightly more general class of domains is the class of inner chordarc domains studied in Geometric Function Theory~\cite{HS, Po, Tu, Va1, Va2}. By definition~\cite{Va1}, a rectifiable Jordan domain $\Y$ is {\it inner chordarc} if there exists a constant  $C$ such that for every pairs of points $y_1, y_2 \in \partial \Y$ there exists an open Jordan arc $\gamma \subset \Y$ with endpoints at $y_1$ and $y_2$ such that the shortest connection from $y_1$ to $y_2$ along $\partial \Y$ has length at most $C\cdot\textnormal{length} (\gamma )$. For example, an inner chordarc domain allows for inward cusps as oppose to Lipschitz domains. According to a result of  V\"ais\"al\"a~\cite{Va1} the inner chordarc condition is equivalent with the requirement that there exists a homeomorphism $\Psi \colon \overline{\Y} \onto \overline{\mathbb D}$, which is $\mathscr C^1$-diffeomorphic in $\Y$, such that the norms of both the gradient matrices $D\Psi$ and $(D\Psi)^{-1}$ are bounded from above.

Surprisingly, the following example shows that, unlike for Lipschitz targets, the answer to Question~\ref{Q2} for $p=2$ is in general negative when the target is only inner chordarc.
 
\begin{example}\label{ex:nodirichglet}
There exists an inner chordarc domain $\Y$ and a homeomorphism $\varphi \colon \partial \mathbb D \onto \partial \Y$ satisfying the Douglas condition~\eqref{eq:douglas} which does not admit a homeomorphic extension $h \colon \overline {\mathbb D} \onto \overline{\Y}$ in $\W^{1,2} (\mathbb D , \Y)$.
\end{example}
In~\cite{AIMO} it was, as a part of studies of mappings with smallest mean distortion,  proved that for $\mathscr C^1$-smooth $\Y$ the  Douglas condition~\eqref{eq:douglas} can be equivalently formulated in terms of  the inverse mapping $\varphi^{-1} \colon \partial \Y \onto \partial \mathbb D$,
\begin{equation}\label{eq:invdouglas}
\int_{\partial \mathbb Y} \int_{\partial \mathbb Y} \big| \log \abs{\varphi^{-1} (\xi) - \varphi^{-1} (\eta) }  \big| \, \abs{\dtext \xi } \,  \abs{\dtext \eta }  < \infty \, .
\end{equation}
It was recently shown that for inner chordarc targets this condition is necessary and sufficient for $\varphi$ to admit a $\W^{1,2}$-homeomorphic extension, see~\cite{KWX}. We extend this result both to cover rectifiable targets and to give a global homeomorphic extension as follows.
\begin{theorem}\label{thm:dirichlet} \textnormal{{\bf ($p=2$)}}
Let $\X$ and $\Y$ be Jordan domains,  $\partial \Y$ being rectifiable. Every $\varphi \colon \partial \X \onto \partial \Y$ satisfying~\eqref{eq:invdouglas} admits a homeomorphic extension $h \colon \cc \to \cc$ of Sobolev class $\W^{1,2}_{loc} (\C, \C)$.
\end{theorem}
Without  the rectifiability of  $\partial \yy$, Question~\ref{Q2} will in general admit a negative answer for all $p \leq 2$. This follows from the following  example of Zhang~\cite{Zh}.
\begin{example}\label{ex:yi}
 There exists a Jordan domain $\Y$ and a homeomorphism $\varphi \colon \partial \mathbb D \onto \partial \Y$ which admits a $\W^{1,2}$-Sobolev extension to $\dd$ but does not admit any homeomorphic extension to $\dd$  in the class $\W^{1,1}(\dd, \mathbb C)$.
\end{example}

We now return to the case when $1 \leq p < 2$. In this case it is natural to ask under which conditions on the domains $\xx$ and $\yy$ does any homeomorphism $\varphi: \partial \xx \onto \partial \yy$ admit a $\W^{1,p}$-Sobolev homeomorphic extension.  Proposition~\ref{Verchota} already implies that this is the case for $\xx = \dd$ and $\yy$ convex. Example~\ref{ex:yi}, however, will imply that this result does not hold in general for nonrectifiable targets $\yy$. A  general characterization is provided by the following two theorems.

\begin{theorem}\label{thm:main} \textnormal{{\bf ($1\le p<2$)}}
Let $\X$ and $\Y$ be Jordan domains in the plane with $\partial \Y$ rectifiable. Let $\varphi \colon \partial \X \onto \partial \Y$ be a given homeomorphism.  Then there is a homeomorphic extension $h \colon \overline{\X} \onto \overline{\Y}$ such that
\begin{enumerate}
\item $h \in \W^{1,1} (\X, \C)$, provided $\partial\X$ is rectifiable, and 
\item $h\in \W^{1,p} (\X, \C)$ for $1<p<2$, provided  $\xx$ has  \emph{$s$-hyperbolic growth} with $s>p-1$. 
\end{enumerate}
\end{theorem}

\begin{definition}\label{def:shyper} Let $\xx$ be a domain in the plane. Choose and fix a point  $x_0 \in \xx$. We say that $\xx$ has  \emph{$s$-hyperbolic growth},  $s \in (0,1)$, if the following condition holds
\begin{equation}\label{cuspCondition}
h_{\xx}(x_0,x) \leq C \left(\frac{\dist(x_0,\partial \xx)}{\dist(x,\partial \xx)}\right)^{1-s} \qquad \text{ for all } x \in \xx \, . 
\end{equation}
Here $h_{\xx}$ stands for the quasihyperbolic metric on $\xx$ and $\dist(x,\partial\xx)$ is the Euclidean distance of $x$ to the boundary. The constant $C$ is allowed to depend on everything except the point $x$.
\end{definition}
It is easily verified that this definition does not depend on the choice of $x_0$. Recall that if $\Omega$ is a domain, the quasihyperbolic metric $h_{\Omega} $ is defined by~\cite{GP}
\begin{equation}
h_{\Omega} (x_1, x_2) = \inf_{\gamma \in \Gamma} \int_\gamma \frac{1}{\dist(x,\partial\xx)}\,  \abs{\dtext x} \, , \qquad x_1 , x_2 \in \Omega 
\end{equation}
where $\Gamma$ is the family of all rectifiable curves in $\Omega$ joining $x_1$ and $x_2$.

Definition~\ref{def:shyper} is motivated by the following example. For $s \in (0,1)$ we consider the Jordan domain $\X_s$ whose boundary is given by the curve
\[\Gamma_s = \{(x,y) \in \cc : -1 \leq x \leq 1,\, y = |x|^s\} \cup \{ z \in \cc : |z-i| = 1,\, \im(z) \geq 1\}.\]

\begin{figure}[h]
\includegraphics[scale=0.4]{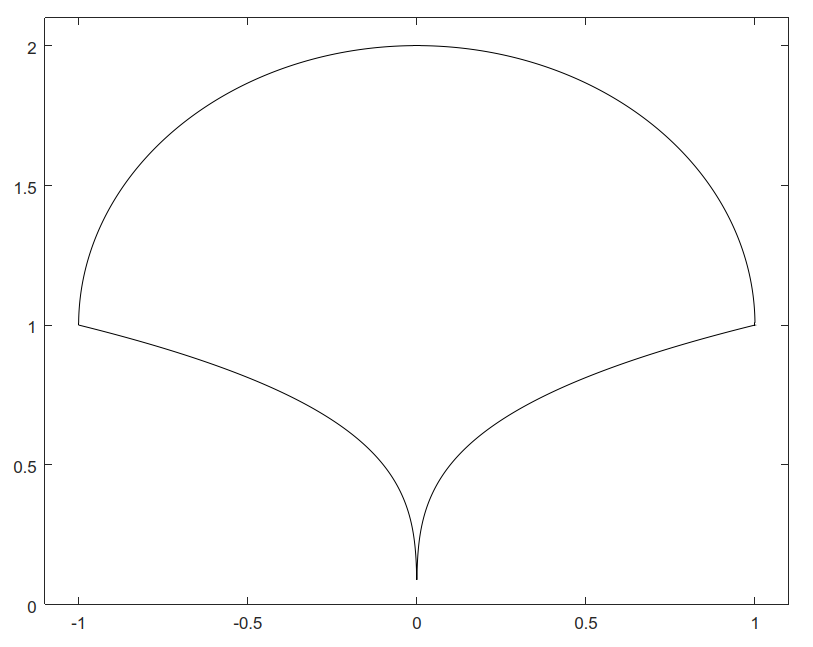}
\caption{The Jordan domain $\xx_s$.}
\label{cuspFig}
\end{figure}

In particular, the boundary of $\X_s$ is locally Lipschitz apart from the origin. Near to  the origin the boundary of $\X_s$ behaves like the graph of the function $|x|^s$. Then one can verify that the boundary of $\X_s$ has $t$-hyperbolic growth for every $t \geq s$. Note that smaller the number $s$ sharper the cusp is.

The results of Theorem~\ref{thm:main} are sharp, as described by the following result.
\begin{theorem}\label{counterExampleThm1} \quad
\begin{enumerate}
\item{There exists a Jordan domain $\xx$ with nonrectifiable boundary and a homeomorphism $\varphi : \partial\xx \to \partial\dd$ such that $\varphi$ does not admit a continuous extension to $\xx$ in the Sobolev class $\W^{1,1}(\xx, \mathbb C)$.}
\item{For every $p \in (1,2)$ there exists a Jordan domain $\xx$ which has $s$-hyperbolic growth, with $p-1 = s$, and a homeomorphism $\varphi : \partial\xx \to \partial\dd$ such that $\varphi$ does not admit a continuous extension to $\xx$ in the Sobolev class $\W^{1,p}(\xx, \mathbb C)$.}
\end{enumerate}
\end{theorem}

To conclude, as  promised earlier,   we extend  our main result to the case where the domains are not simply connected. The following  generalization of Theorem \ref{thm:main} holds.

\begin{theorem}\label{thm:multiply} Let $\xx$ and $\yy$ be multiply connected Jordan domains with $\partial \Y$ rectifiable. Let $\varphi \colon \partial \X \onto \partial \Y$ be a given homeomorphism which maps the outer boundary component of $\X$ to the outer boundary component of $\Y$.  Then there is a homeomorphic extension $h \colon \overline{\X} \onto \overline{\Y}$ such that
\begin{enumerate}
\item $h \in \W^{1,1} (\X, \C)$, provided $\partial\X$ is rectifiable, and 
\item $h\in \W^{1,p} (\X, \C)$ for $1<p<2$, provided  $\xx$ has  \emph{$s$-hyperbolic growth} with $s>p-1$. 
\end{enumerate}
\end{theorem}

\subsection*{Acknowledgements}
We thank Pekka Koskela for posing the main question of this paper to us.

\section{Preliminaries}\label{sec:pre}
\subsection{The Dirichlet problem}
Let  $\Omega $ be a bounded domain in the complex plane.  A function $u \colon \Omega \to \mathbb R$  in the Sobolev class $\W^{1,p}_{\loc} (\Omega)$, $1<p<\infty$,  is called {\it $p$-harmonic} if
\begin{equation}\label{eq:eijoo}\Div \abs{\nabla u}^{p-2} \nabla u =0. \end{equation}
We call $2$-harmonic functions simply \emph{harmonic}.
 
There are two  formulations of the Dirichlet boundary value problem for  the $p$-harmonic equation~\eqref{eq:eijoo}. We first consider  the variational formulation.
\begin{lemma}
Let $u_\circ \in \W^{1,p}(\Omega)$ be a given Dirichlet data. There exists precisely one function $u\in u_\circ +  \W_\circ^{1,p}(\Omega)$ which minimizes the $p$-harmonic energy:
\[\int_\Omega \abs{\nabla u}^p =\inf \left\{ \int_\Omega \abs{\nabla w}^p\colon w\in u_\circ +  \W_\circ^{1,p}(\Omega) \right\}.\]
\end{lemma}
The variational formulation coincides with the classical formulation of the Dirichlet problem.

\begin{lemma}\label{proexist} Let $\Omega \subset \C$ be a bounded Jordan domain and $u_\circ \in \W^{1,p}(\Omega) \cap \mathscr C (\overline{\Omega})$. Then there exists a unique $p$-harmonic function $u\in \W^{1,p}(\Omega) \cap \mathscr C (\overline{\Omega})$ such that $u_{|_{\partial \Omega}}=u_{\circ |_{\partial \Omega}}$.
\end{lemma}

For a reference for proofs of these facts we refer to~\cite{IKOapprox}.

 \subsection{The Rad\'o-Kneser-Choquet Theorem}
 \begin{lemma}\label{lem:RKC}
Consider a  Jordan domain $\mathbb X \subset \C$ and a bounded convex domain $\mathbb Y \subset \mathbb C$. Let $h \colon \partial \mathbb X \onto \partial \mathbb Y$ be a homeomorphism  and $H \colon \mathbb U \to \C$ denote its harmonic extension. Then $H$ is a $\mathscr C^\infty$-diffeomorphism of $\mathbb X$ onto $\mathbb Y$.
\end{lemma}
For the proof of this lemma we refer to~\cite{Dub, IOrado}. The following $p$-harmonic analogue of the Rad\'o-Kneser-Choquet Theorem is due to Alessandrini and Sigalotti~\cite{AS}, see also~\cite{IOsimply}.
 
\begin{proposition}
Let $\X$  be a Jordan domain  in $\mathbb C$, $1<p< \infty$, and $h=u+iv \colon \overline{\X } \to \mathbb C$ be a continuous mapping whose coordinate functions  are $p$-harmonic. Suppose that $\Y$ is convex and that $h \colon \partial \X \onto \partial \Y$ is a homeomorphism. Then $h$ is a diffeomorphism from $\X$ onto $\Y$.
\end{proposition}

\subsection{Sobolev homeomorphic extensions onto a Lipschitz target}\label{sec:liptriv}
Combining the results in this section allows us to easily solve Question~\ref{Q2} for convex targets.
\begin{proposition}\label{pr:blah}
Let $\X$ and $\Y$ be Jordan domains in the plane with $\Y$ convex, and let $p$ be given with $1<p< \infty$. Suppose that $\varphi \colon \partial \X \onto \partial \Y$ is a homeomorphism. Then   
there exists a continuous $g\colon \overline{\X} \to \C$  in $\W^{1,p} (\X, \C)$ such that $g(x)=\varphi(x)$ on $\partial \X$ if and only if there exists a homeomorphism  $h\colon \overline{\X} \to \overline{\Y}$  in $\W^{1,p} (\X, \C)$ such that $h(x)=\varphi(x)$ on $\partial \X$.
\end{proposition}
Now, replacing the convex $\Y$ by a Lipschitz domain offers no challenge. Indeed, this follows from  a global bi-Lipschitz change of  variables $\Phi \colon \mathbb C \to \mathbb C$ for which $\Phi (\Y)$ is the unit disk. If the domain in Proposition~\ref{pr:blah} is the unit disk $\mathbb D$, then the existence of a finite $p$-harmonic extension can be characterized in terms of a Douglas type condition. If $1<p<2$, then such an extension exists for an arbitrary boundary homeomorphism  (Proposition~\ref{Verchota})  and  if $2 \le p < \infty$ the extension exists if and only the boudary homeomorphism $\varphi \colon \partial \mathbb D \onto \partial \Y$ satisfies the following condition,
\begin{equation}\label{eq:pdouglas}
\int_{\partial \mathbb D} \int_{\partial \mathbb D} \left| \frac{\varphi (\xi) - \varphi (\eta)}{ \xi - \eta }\right|^p \abs{\dtext \xi } \,  \abs{\dtext \eta }  < \infty \, .
\end{equation}
For the proof of the latter fact we refer to~\cite[p. 151-152]{Stb}.

\subsection{A Carleson measure and the Hardy space $H^p$}
 Roughly speaking, a Carleson measure on a domain $\mathbb G$ is a measure that does not vanish at the boundary of $\mathbb G$ when compared to the Hausdorff $1$-measure on $\partial \mathbb G$.  We will need the notion of Carleson measure only on the unit disk $\mathbb D$.
 \begin{definition}\label{def:carleson}
 Let $\mu$  be a Borel measure on $\mathbb D$. Then $\mu$ is a {\it Carleson measure} if there is a constant $C>0$ such that
 \[\mu (S_\epsilon(\theta)) \le C \epsilon\]
 for every $\epsilon >0$. Here
  \[ S_\epsilon(\theta) = \{r e^{i\alpha} : 1 - \epsilon < r < 1, \theta - \epsilon < \alpha < \theta + \epsilon\} \, . \]
 \end{definition}
Carleson measures have many applications in harmonic analysis. A celebrated result by L. Carleson~\cite{Ca}, also see Theorem 9.3 in~\cite{Duhp}, tells us that a Borel measure $\mu$ on $\mathbb D$ is  a bounded Carleson measure if and only if the injective mapping from the Hardy space $H^p (\mathbb D)$ into a the measurable space $ L^p_\mu (\mathbb D)$ is bounded:
\begin{proposition}\label{pro:carleson}
Let $\mu$ be a Borel measure on the unit disk $\mathbb D$. Let $0<p<\infty$.  Then in order that there exist a constant $C>0$ such that
\[\left( \int_{\mathbb D} \abs{f(z)}^p \, d \mu (z)\right)^\frac{1}{p} \le C ||f||_{H^p (\mathbb D)} \quad \textnormal{for all } f \in H^p (\mathbb D) \]
it is necessary and sufficient that $\mu$ be a Carleson measure.
\end{proposition}
Recall that the Hardy space $H^p(\mathbb D)$, $0<p<\infty$, is the class of holomorphic functions $f$ on the unit disk satisfying
\[||f||_{H^p(\mathbb D)}:=\sup_{0 \le r <1 } \left(\frac{1}{2\pi} \int_0^{2\pi }  \abs{f(r e^{i \theta})}^p \, d \theta   \right)^\frac{1}{p} < \infty \, . \]
Note that $||\cdot||_{H^p(\mathbb D)}$ is a norm when $p \ge 1$, but not when $0<p<1$.

\section{Sobolev integrability of the harmonic extension}

At the end of this section we prove our main result in the simply connected case, Theorem \ref{thm:main}. The proof will be based on a suitable reduction of the target domain to the unit disk, and the following auxiliary result which concerns the regularity of harmonic extensions. 
\begin{theorem}\label{harmonicThm}
Let $\xx$ be a Jordan domain and $\varphi : \partial \xx \to \partial \dd$ be an arbitrary homeomorphism. Let $h$ denote the harmonic extension of $\varphi$ to $\xx$, which is a homeomorphism from $\bar{\xx}$ to $\bar{\dd}$. Then the following hold.
\begin{enumerate}
\item{If the boundary of $\xx$ is rectifiable, then $h \in \W^{1,1}(\xx, \C)$.}
\item{If $\xx$ has $s$-hyperbolic growth, then $h \in \W^{1,p}(\xx, \C)$ for $p = s-1$.}
\end{enumerate}
\end{theorem}
This theorem will be a direct corollary of the following theorem and the two propositions after it.

\begin{theorem}\label{harmonicThmV2}
Let $\xx$ be a Jordan domain, and denote by $g : \dd \to \xx$ a conformal map onto $\xx$. Let $1 \leq p < 2$. Suppose that the condition
\begin{equation}\label{integrableConfMap}
\sup_{\omega \in \partial \dd} \int_{\dd} \frac{|g'(z)|^{2-p}}{|\omega - z|^p} \, dz \leq M < \infty
\end{equation}
holds. Then the harmonic extension $h : \xx \to \dd$ of any boundary homeomorphism $\varphi : \partial \xx \to \partial \dd$ lies in the Sobolev space $\W^{1,p}(\xx, \C)$, with the estimate
\begin{equation}\label{harmonicpEnergyEstim}
||h||_{\W^{1,p}(\xx, \C)} \leq c M. 
\end{equation}
\end{theorem}

\begin{proposition}\label{h1Conformal}
Let $\xx$ be a Jordan domain with rectifiable boundary. Then the condition \eqref{integrableConfMap} holds with $p=1$.
\end{proposition}

\begin{proposition}\label{pConformal}
Let $\xx$ be a Jordan domain which has $s$-hyperbolic growth, with $s \in (0,1)$. Then condition \eqref{integrableConfMap} holds for all $p > 1$ with $p-1 < s$.
\end{proposition}

\begin{proof}[Proof of Theorem \ref{harmonicThmV2}]
First, since $\X$ is a Jordan domain  according to the classical Carath\'eodory's theorem the conformal mapping $g \colon  \mathbb D \to \mathbb X$ extends continuously to a homeomorphism from the unit circle onto  $\partial \mathbb X$. Second, since a conformal change of variables preserves harmonicity, we find that the map $H := h \circ g \colon  \dd \to \dd$ is a harmonic extension of the boundary homeomorphism $\psi := \varphi \circ g\vert_{\partial \dd}$. 

We will now assume that $H$ is smooth up to the boundary of $\dd$. The general result will then follow by an approximation argument. Indeed, for each $r < 1$, we may take the preimage of the disk $B(0,r)$ under $H$, and letting $\psi_r: \dd \to H^{-1}(B(0,r))$ be the conformal map onto this preimage we may define $H_r := H \circ \psi_r$. Then $H_r$ is harmonic, smooth up to the boundary of $\dd$, and will converge to $H$ locally uniformly along with its derivatives as $r \to 1$. Hence the general result will follow once we obtain uniform estimates for the Sobolev norm under the assumption of smoothness up to the boundary.

The harmonic extension $H := h \circ g \colon  \dd \to \dd$  of  $\psi := \varphi \circ g\vert_{\partial \dd}$ is given by the Poisson integral formula~\cite{Dub},
\[(h \circ g) (z)=H(z) = \frac{1}{2 \pi} \int_{\partial \mathbb D} \frac{1-\abs{z}^2}{\abs{z- \omega}} \psi (\omega) \, d \omega \, . \]

Differentiating this, we find the formula
\begin{equation*}
(h \circ g)_z = \int_{\partial \dd} \frac{\psi(\omega)}{(z-\omega)^2}\, d\omega
= \int_0^{2\pi} \frac{\psi(e^{it})}{(z-e^{it})^2} i e^{it}\, dt
= \int_0^{2\pi} \frac{\psi'(e^{it})}{z-e^{it}} i e^{it}\, dt,
\end{equation*}
where we have used integration by parts to arrive at the last equality. The change of variables formula now gives
\begin{align*}
\int_{\xx} |h_z(\tilde{z})|^p \, d\tilde{z} &= \int_{\dd} |(h \circ g)_z(z)|^p |g'(z)|^{2-p} \, dz
\\&= \int_{\dd} \left|\int_0^{2\pi} \frac{\psi'(e^{it})}{z-e^{it}} i e^{it}\, dt\right|^p |g'(z)|^{2-p} \, dz,
\end{align*}
We now apply Minkowski's integral inequality to find that
\begin{align*}
&\left(\int_{\dd} \left|\int_0^{2\pi} \frac{\psi'(e^{it})}{z-e^{it}} i e^{it}\, dt\right|^p |g'(z)|^{2-p} \, dz\right)^{\frac{1}{p}} \\&\qquad \leq
\int_0^{2\pi} |\psi'(e^{it})| \left(\int_{\dd} \frac{|g'(z)|^{2-p}}{|z-e^{it}|^p} \, dz\right)^{\frac{1}{p}}  \, dt
\\&\qquad\leq M\int_0^{2\pi} |\psi'(e^{it})| \, dt
\\&\qquad= 2\pi M
\end{align*}
This gives the uniform bound $||h_z||_{L^p(\xx)} \leq 2\pi M$. An analogous estimate for the $L^p$-norm of $h_{\bar{z}}$ now proves the theorem.
\end{proof}
\begin{proof}[Proof of Proposition \ref{h1Conformal}]
Since $\partial\xx$ is rectifiable, the derivative $g'$ of a conformal map from $\dd$ onto $\xx$ lies in the Hardy space $H^1(\mathbb D)$ by Theorem 3.12 in~\cite{Duhp}. By rotational symmetry it is enough to verify condition \eqref{integrableConfMap} for $\omega = 1$ and $g : \dd \to \xx$ an arbitrary conformal map. By Proposition~\ref{pro:carleson}, it suffices to verify that the measure $\mu(z) = \frac{dz}{|1-z|}$ is a Carleson measure, see Definition~\ref{def:carleson}, to obtain the estimate
\[\int_{\dd} \frac{|g'(z)|}{|1 - z|} \, dz \leq C ||g'||_{H^1(\mathbb D)},\]
which will imply that the proposition holds. Let us hence for each $\epsilon$ define the set $S_\epsilon(\theta) = \{r e^{i\alpha} : 1 - \epsilon < r < 1, \theta - \epsilon < \alpha < \theta + \epsilon\}$. We then estimate for small $\epsilon$ that
\begin{equation*}
\mu(S_\epsilon(0)) \leq \mu(B(1,2\epsilon))
= \int_{B(1,2\epsilon)} \frac{dz}{|1-z|}
= \int_0^{2\pi} \int_0^{2\epsilon} \frac{1}{r}\, r \, dr d\alpha
= 4\pi \epsilon.
\end{equation*}
It is clear that for any other angles $\theta$ the $\mu$-measure of $S_\epsilon(\theta)$ is smaller than for $\theta = 0$. Hence $\mu$ is a Carleson measure and our proof is complete.
\end{proof}

\begin{proof}[Proof of Proposition \ref{pConformal}]
Recall that $g$ denotes the conformal map from $\dd$ onto $\xx$. Since $\xx$ has $s$-hyperbolic growth, we may apply Definition \ref{cuspCondition} with $x_0 = g(0)$ to find the estimate
\begin{equation}\label{hyperbolicEstim1}
h_{\xx}(g(0),g(z)) \leq C \left(\frac{1}{\dist(g(z),\partial \xx)}\right)^{1-s} \qquad \text{ for all } z \in \dd,
\end{equation}
Since $\xx$ is simply connected, the quasihyperbolic distance is comparable to the hyperbolic distance $\rho_{\xx}$. By conformal invariance of the hyperbolic distance we find that
\[C_1 h_{\xx}(g(0),g(z)) \geq \rho_{\xx}(g(0),g(z)) = \rho_{\dd}(0,z) = \log \frac{1}{1-|z|^2}.\]
Now by the Koebe $\frac14$-theorem we know that the expression $\dist(g(z),\partial \xx)$ is comparable to $(1-|z|)|g'(z)|$ by a universal constant. Combining these observations with \eqref{hyperbolicEstim1} leads to the estimate
\begin{equation*}\label{hyperbolicEstim2}
\log \frac{1}{1-|z|^2} \leq C \left(\frac{1}{(1-|z|)|g'(z)|}\right)^{1-s},
\end{equation*}
which we transform into
\begin{equation}\label{hyperbolicEstim3}
|g'(z)| \leq \frac{C}{(1-|z|) \log^{1/(1-s)} \frac{1}{1-|z|}},
\end{equation}
Let us denote $\beta = (2-p)/(1-s)$ so that $\beta > 1$ by assumption. We now apply the estimate \eqref{hyperbolicEstim3} to find that
\begin{align}\label{integralLogEstim}\int_{\dd} \frac{|g'(z)|^{2-p}}{|1 - z|^p} \, dz \leq C&\int_{\dd \setminus \frac12 \dd} \frac{1}{(1-|z|)^{2-p}|1 - z|^p \log^{\beta} \frac{1}{1-|z|}} \, dz \\ \nonumber & + \int_{\frac12 \dd} \frac{|g'(z)|^{2-p}}{|1 - z|^p} \, dz.\end{align} 
It is enough to prove that the quantity on the right hand side above is finite as then rotational symmetry will imply that the estimate \eqref{integrableConfMap} holds for all $\omega$. The second term is easily seen to be finite, as the integrand is bounded on the set $\frac12 \dd$. To estimate the first integral we will cover the annulus $\dd \setminus \frac12\dd$ by three sets $S_1,S_2$ and $S_3$ defined by
\begin{align*}S_1 &= \{1 + r e^{i\theta} : r \leq 3/4, \, 3\pi/4 \leq \theta \leq 5\pi/4\}
\\ S_2 &= \{(x,y) \in \dd : -1/\sqrt{2} \leq y \leq 1/\sqrt{2}, \, x \leq 1, \, x \geq 1 - |y|\}
\\ S_3 &= \{ r e^{i\theta} : 1/2 \leq r \leq 1, \, \pi/4 \leq \theta \leq 7\pi/4\}
\end{align*}
See Figure \ref{Sfig} for an illustration of these sets. Since the sets $S_1,S_2$ and $S_3$ cover the set in question, it will be enough to see that the first integral on the right hand side of equation \eqref{integralLogEstim} is finite when taken over each of these sets.
\begin{figure}[h]
\includegraphics[scale=0.3]{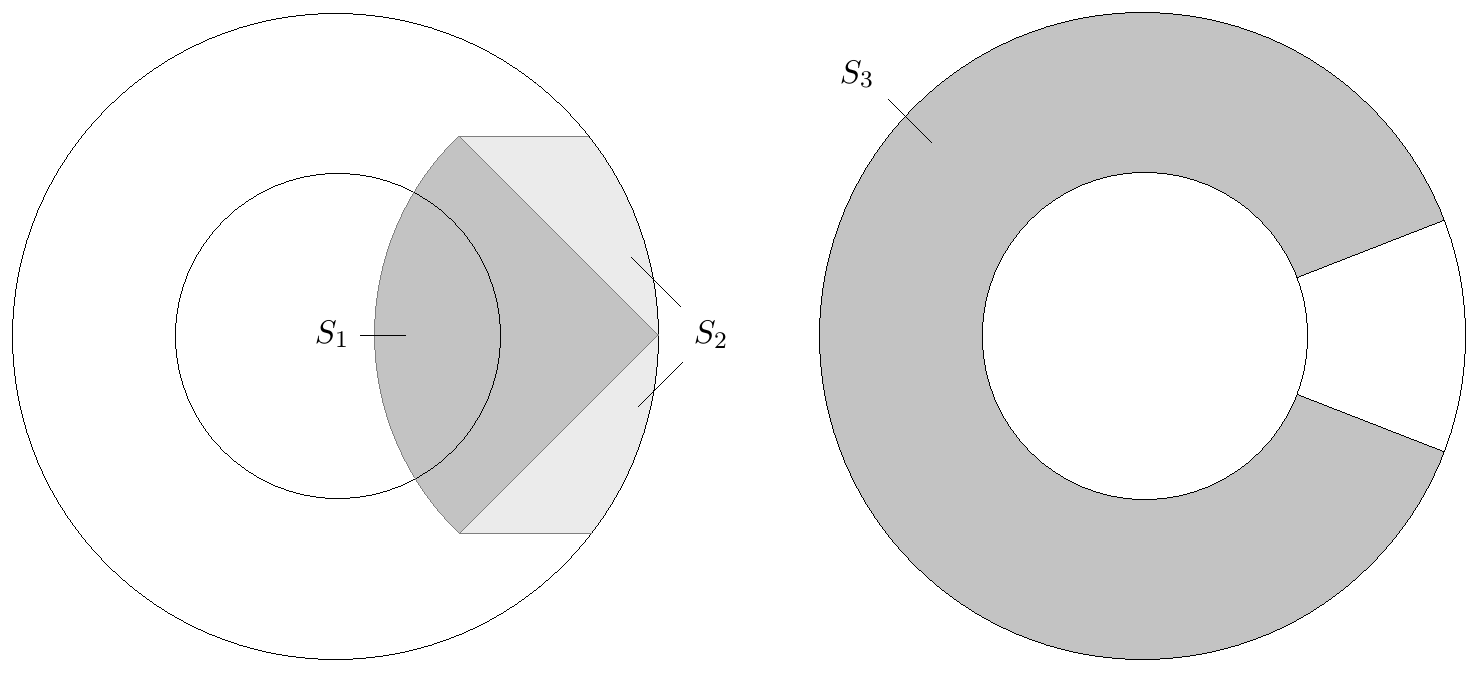}
\caption{The sets $S_i$, $i = 1,2,3$.}
\label{Sfig}
\end{figure}
On the set $S_1$, one may find by geometry that the estimate $1-|z| \geq c|1-z|$ holds for some constant $c$. Hence we may apply polar coordinates around the point $z = 1$ to find that
\[\int_{S_1} \frac{1}{(1-|z|)^{2-p}|1 - z|^p \log^{\beta} \frac{1}{1-|z|}} \, dz \leq C\int_{3\pi/4}^{5\pi/4} \int_0^{3/4} \frac{1}{r \log^{\beta} \frac{1}{r}} \, dr d\theta < \infty.\]
On the set $S_3$, the expression $|1-z|$ is bounded away from zero. Hence bounding this term and the logarithm from below and changing to polar coordinates around the origin yields that
\begin{align*}\int_{S_3}& \frac{1}{(1-|z|)^{2-p}|1 - z|^p \log^{\beta} \frac{1}{1-|z|}} \, dz
\leq C \int_{\pi/4}^{7\pi/4} \int_{1/2}^1 \frac{r}{(1-r)^{2-p}} \, dr d\theta < \infty. 
\end{align*}
On the set $S_2$, we change to polar coordinates around the origin. For each angle $\theta$, we let $R_\theta$ denote the intersection of the ray with angle $\theta$ starting from the origin and the set $S_2$. On each such ray, we find that the expression $|1-z|$ is comparable to the size of the angle $\theta$. Since $1-|z| < |1-z|$, we may also replace $1-|z|$ by $|1-z|$ inside the logarithm, in total giving us the estimate
\begin{equation}\label{thetaEstim}
\frac{1}{|1-z|^p \log^{\beta} \frac{1}{1-|z|}} \leq \frac{C}{|\theta|^p \log^{\beta} \frac{1}{|\theta|}}, \qquad z \in R_\theta.
\end{equation}
On each of the segments $R_\theta$ and small enough $\theta$, the modulus $r = |z|$ ranges from a certain distance $\rho(\theta)$ to $1$. This distance is found by applying the sine theorem to the triangle with vertices $0, 1$ and $\rho(\theta) e^{i\theta}$, giving us the equation
\[\frac{\rho(\theta)}{\sin(\pi/4)} = \frac{1}{\sin(\pi - \pi/4 - \theta)} = \frac{1}{\sin(\pi/4 + \theta)}.\]
From this one finds that the expression $1 - \rho(\theta) = \frac{\sin(\pi/4 + \theta) - \sin(\pi/4)}{\sin(\pi/4+\theta)}$, which also denotes the length of the segment $R_\theta$, is comparable to $|\theta|$. Using this and \eqref{thetaEstim} we now estimate that
\begin{align*}\int_{S_2}& \frac{1}{(1-|z|)^{2-p}|1 - z|^p \log^{\beta} \frac{1}{1-|z|}} \, dz
\\ &\leq C \int_{-\pi/4}^{\pi/4} \frac{1}{|\theta|^p \log^{\beta} \frac{1}{|\theta|}} \int_{\rho(\theta)}^1 \frac{1}{(1-r)^{2-p}} \, dr d\theta
\\ &= C \int_{-\pi/4}^{\pi/4} \frac{1}{|\theta|^p \log^{\beta} \frac{1}{|\theta|}}  \frac{(1-\rho(\theta))^{p-1}}{p-1} \, dr d\theta
\\ &\leq C \int_{-\pi/4}^{\pi/4} \frac{1}{|\theta| \log^{\beta} \frac{1}{|\theta|}}  \, dr d\theta
\\&< \infty. 
\end{align*}
This finishes the proof.
\end{proof}
\begin{proof}[Proof of Theorem~\ref{thm:main}]
Since $\yy$ is a rectifiable Jordan domain, there exists a constant speed parametrization $\gamma : \partial \dd \to \partial \yy$. Such a parametrization is then automatically a Lipschitz embedding of $\partial \dd$ to $\cc$, and hence Theorem \ref{thm:lip} implies that there exists a homeomorphic Lipschitz extension $G: \bar{\dd} \to \bar{\yy}$ of $\gamma$.

Let now $\varphi : \partial\xx \to \partial\yy$ be a given boundary homeomorphism. We define a boundary homeomorphism $\varphi_0 : \partial \xx \to \partial \dd$ by setting $\varphi_0 := \varphi \circ \gamma^{-1}$. Let $h_0$ denote the harmonic extension of $\varphi_0$ to $\xx$, so that by the RKC-theorem (Lemma~\ref{lem:RKC}) the composed map $h := G \circ h_0 : \bar{\xx} \to \bar{\yy}$ gives a homeomorphic extension of the boundary map $\varphi$. If the map $h_0$ lies in the Sobolev space $\W^{1,p}(\xx, \C)$, then so does the map $h$ since the Sobolev integrability is preserved under a composition by a Lipschitz map. Hence Theorem~\ref{thm:main} now follows from Theorem~\ref{harmonicThm}.
\end{proof}

\section{Sharpness of Theorem \ref{thm:main}}
In this section we prove Theorem \ref{counterExampleThm1}. We handle the two claims of this theorem separately.\\\\
\textbf{Example (1).} In this example we construct a nonrectifiable Jordan domain $\xx$ and a boundary map $\varphi:\partial \xx \to \partial \dd$ which does not admit a continuous extension in the Sobolev class $\W^{1,1}(\xx,\cc)$. The domain $\xx$ will be defined as the following ``spiral'' domain.

Let $R_k$, $k = 1,2,3,\ldots$, be a set of disjoint rectangles in the plane such that their bottom side lies on the $x$-axis. Each rectangle has width $w_k$ so that $\sum_{k=1}^{\infty} w_k < \infty$ and the rectangles are sufficiently close to each other so that the collection stays in a bounded set. The heights $h_k$ satisfy $\lim_{k \to \infty} h_k = 0$ and $\sum_{k=1}^{\infty} h_k = \infty$.

We now join these rectangles into a spiral domain as in Figure \ref{figSpiral}, and add a small portion of boundary to the bottom end of $R_1$. The exact way these rectangles are joined is not significant, but it is clear that it may be done in such a way as to produce a nonrectifiable Jordan domain $\xx$ for any sequence of rectangles $R_k$ as described above.

\begin{figure}[h]
\includegraphics[scale=0.3]{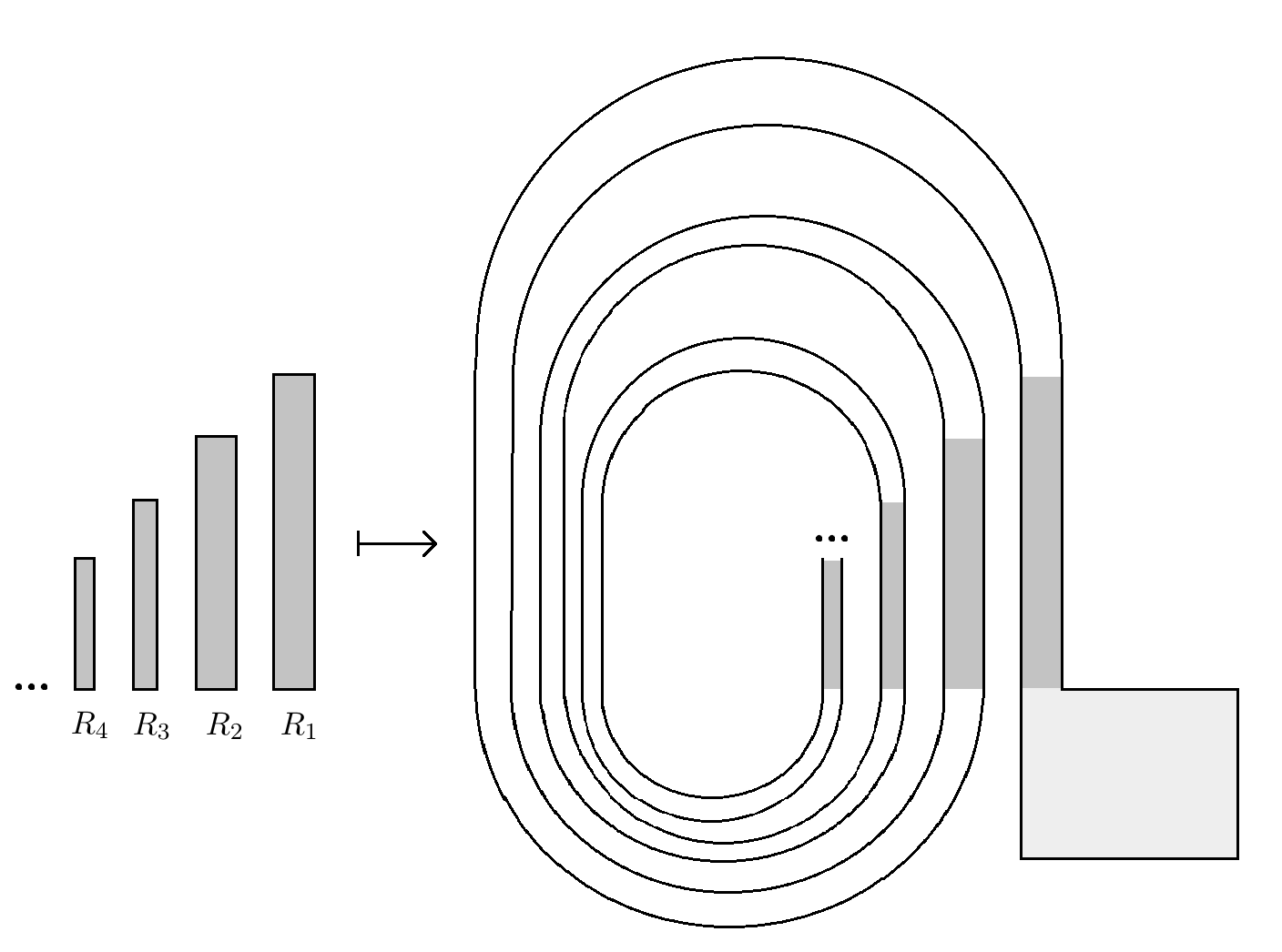}
\caption{The rectangles $R_k$ joined into the spiral domain $\xx$.}
\label{figSpiral}
\end{figure}
Let us now define the boundary homeomorphism $\varphi$. The map $\varphi$ shall map the ``endpoint" (i.e. the point on the $x$-axis to which the rectangles $R_k$ converge) of the spiral domain $\xx$ to the point $1 \in \partial \dd$. Furthermore, we choose disjoint arcs $A_k^+$ on the unit circle so that the endpoints of $A_k^+$ are given by $e^{i\alpha_k}$ and $e^{i\beta_k}$ with
\[\pi/2 > \alpha_1 > \beta_1 > \alpha_2 > \beta_2 > \cdots\]
and $\lim_{k\to\infty} \alpha_k = 0$. We mirror the arcs $A_k^+$ over the $x$-axis to produce another set of arcs $A_k^-$. The arcs are chosen in such a way that the minimal distance between $A_k^+$ and $A_k^-$ is greater than a given sequence of numbers $d_k$ with $\lim_{k\to\infty} d_k = 0$. It is clear that for any such sequence we may make a choice of arcs as described here.

We now define $\varphi$ to map the left side of the rectangle $R_k$ to the arc $A_k^+$, and the right side to $A_k^-$. On the rest of the boundary $\partial \xx$ we may define the map $\varphi$ in an arbitrary way to produce a homeomorphism $\varphi : \partial \xx \to \partial \dd$.

Let now $H$ be a continuous $\W^{1,1}$-extension of $\varphi$. Let $I_k$ denote any horizontal line segment with endpoints on the vertical sides of $R_k$. Then by the above construction, $H$ must map the segment $I_k$ to a curve of length at least $d_k$, as this is the minimal distance between $A_k^+$ and $A_k^-$. Hence we find that
\[\int_{R_k} |DH| dz \geq \int_0^{h_k} d_k dz = h_k d_k.\]
Summing up, we obtain the estimate
\[\int_\xx |DH| dz \geq \sum_{k=1}^\infty h_k d_k.\]
We may now choose, for example, $h_k = 1/k$ and $d_k = 1/\log (1+k)$ to make the above sum diverge, showing that $H$ cannot belong to $\W^{1,1}(\xx, \C)$. This finishes the proof.\\\\ 
\textbf{Example (2).} Let $1 < p < 2$. Here we construct a Jordan domain $\xx$ whose boundary has $(p-1)$-hyperbolic growth and a boundary map $\varphi:\partial \xx \to \partial \dd$ which does not admit a continuous extension in the Sobolev class $\W^{1,p}(\xx,\cc)$. In fact, this domain may be chosen as the domain $\xx_s$ defined after Definition \ref{cuspCondition} for $s = p-1$.

The construction of the boundary map $\varphi$ is as follows.

We set $\varphi(0) = 1$. Furthermore, we choose two sequences of points $p_k^{+}$ and $p_k^{-}$ belonging to the graph $\{(x,|x|^s) : -1 \leq x \leq 1\}$ as follows. The points $p_k^{+}$ all have positive $x$-coordinates, their $y$-coordinates are decreasing in $k$ with limit zero and the difference between the $y$-coordinates of $p_{k-1}^{+}$ and $p_k^+$ is comparable to a number $\epsilon_k$, for which
\[\sum_{k=1}^{\infty} \epsilon_k < \infty.\]
In fact, for any sequence of numbers $\epsilon_k$ satisfying the above conditions one may choose a corresponding sequence $p_k^+$. We then let $p_k^-$ be the reflection of $p_k^+$ along the $y$-axis.

Similarly, we choose points $a_k^+$ on the unit circle, so that $a_k^+ = e^{i \theta_k}$ for a sequence of angles $\theta_k > 0$ decreasing to zero. Letting $a_k^-$ be the reflection of $a_k^+$ along the $x$-axis, we choose the sequence in such a way that the line segment between $a_k^+$ and $a_k^-$ has length $d_k$ for some decreasing sequence $d_k$ with $\lim_{k\to \infty} d_k = 0$. Again, any such sequence $d_k$ gives rise to a choice of points $a_k$.

Let $\Gamma_k^{+}$ denote the part of the boundary of $\xx_s$ between $p_{k-1}^{+}$ and $p_k^{+}$. We define the map $\varphi$ to map $\Gamma_k^{-}$ to the arc of the unit circle between $a_{k-1}^{-}$ and $a_{k}^{-}$ with constant speed. We define $\Gamma_k^-$ and $\varphi \vert_{\Gamma_k^-}$ similarly.
\begin{figure}[h]
\includegraphics[scale=0.37]{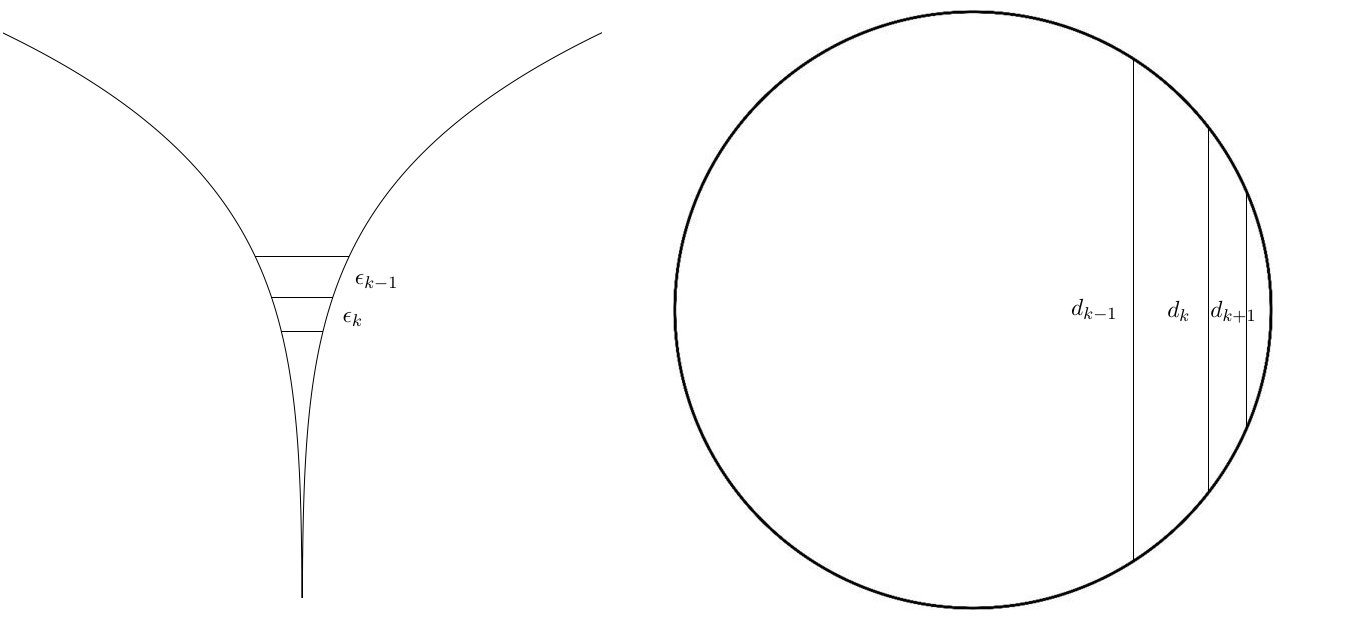}
\caption{The portions of height $\epsilon_k$ get mapped onto slices with side length $d_k$.}
\end{figure}
Let now $H$ denote any continuous $\W^{1,p}$-extension of $\varphi$ to $\xx$. By the above definition, any horizontal line segment with endpoints on $\Gamma_k^+$ and $\Gamma_k^-$ is mapped into a curve of length at least $d_k$ under $H$. Such a line segment is of length at most the distance of $p_{k-1}^+$ to $p_{k-1}^-$, a distance which is comparable to $\left(\sum_{j = k}^{\infty} \epsilon_j \right)^{1/s}$. If $S_k$ denotes the union of all the horizontal line segments between $\Gamma_k^+$ and $\Gamma_k^-$, this gives the estimate
\[\int_{S_k} |DH|^p dz \geq \frac{\left(\int_{S_k} |DH| dz\right)^p}{|S_k|^{p-1}} \geq \frac{c \left(\int_0^{\epsilon_k} d_k dy\right)^p}{\epsilon_k^{p-1} \left(\sum_{j = k}^{\infty} \epsilon_j \right)^{(p-1)/s} } = \frac{c d_k^p \epsilon_k}{\sum_{j = k}^{\infty} \epsilon_j}\]
Let now, for example, $\epsilon_k = 1/k^2$. Then $\sum_{j = k}^{\infty} \epsilon_j$ is comparable to $1/k$, so by summing up we obtain the estimate
\begin{equation}\label{sumEstim1}\int_{\bigcup_k S_k} |DH|^p dz \geq c\sum_{k=1}^\infty \frac{d_k^p}{k}.\end{equation}
Choosing a suitably slowly converging sequence $d_k$ such as $d_k = (\log (1+k))^{-1/p}$, we find that the right hand side of \eqref{sumEstim1} diverges. It follows that $H$ cannot lie in the Sobolev space $\W^{1,p}(\xx_s, \C)$, which completes our proof.
\section{The case $p = 2$}
In this section we address Theorem \ref{thm:dirichlet} as well as Examples \ref{ex:nodirichglet} and \ref{ex:yi}.\\\\
\textbf{Example \ref{ex:nodirichglet}.} For this example, let first $\Phi_\tau$ for any $\tau \in (0,1]$ denote the conformal map
\[\Phi_\tau(z) = \log^{-\tau}\left(\frac{1-z}{3}\right)\]
defined on the unit disk and having target $\yy_\tau := \Phi_\tau(\dd)$. In fact, the domain $\yy_\tau$ is a domain with smooth boundary apart from one point at which it has an outer cusp of degree $\tau/(1+\tau)$ (i.e. it is bilipschitz-equivalent with the domain $\xx_{\tau/(1+\tau)}$ as pictured in Figure \ref{cuspFig}).

Since $\Phi_\tau$ is conformal and maps the unit disk into a set of finite measure, it lies in the Sobolev space $\W^{1,2}(\dd, \C)$. However, it does not admit a homeomorphic extension to the whole plane in the Sobolev class $\W^{1,2}_{loc}(\cc)$. The reason for this is that there is a modulus of continuity estimate for any homeomorphism in the Sobolev class $\W^{1,2}_{loc}(\cc)$. Indeed, let $\omega (t)$ denote the the modulus of continuity of $g \colon \C \to \C$; that is,
\[\omega (t) = \underset{B(z,t)}{\osc}    g = \sup \{ \abs{g(x_1)- g(x_2)} \colon x_1, x_2 \in B(z,t)\} \, .  \]
If $g$ is a homeomorphism in $\W_{\loc}^{1,2} (\C, \C)$, then
\begin{equation}\label{eq456}\int_0^r \frac{\omega(t)^2}{t} dt < \infty.\end{equation}
\begin{proof}[Proof of~\eqref{eq456}] Since $g$ is a homeomorphism we have
\[\underset{B(z,t)}{\osc}     g \le\underset{\partial B(z,t)}{\osc}    g \, . \]
According to Sobolev's inequality on spheres for almost every $t>0$ we obtain
\[ \underset{\partial B(z,t)}{\osc}     g \le C \int_{\partial B(z,t)} \abs{Dg} \, .  \]
These together with H\"older's inequality  imply
\[\omega (t) =  \underset{B(z,t)}{\osc}   \le \underset{\partial B(z,t)}{\osc}    g\le   C\left(t \, \int_{\partial B(z,t)} \abs{Dg}^2 \right)^\frac{1}{2}  \]
and, therefore, for almost every $t>0$ we have
\[\frac{\omega (t)^2}{t} \le C \int_{\partial B(z,t)} \abs{Dg}^2 \, .  \]
Integrating this from $0$ to $r>0$, the claim~\eqref{eq456} follows.
\end{proof}
Now, since the map $\Phi_\tau$ for $\tau \leq 1$ does not satisfy the modulus of continuity estimate~\eqref{eq456} at the boundary point $z = 1$, it follows that it is not possible to extend $\Phi_\tau$ even locally as a $\W^{1,2}$-homeomorphism around the point $z = 1$.

To address the exact claim of Example \ref{ex:nodirichglet}, we now define an embedding $\varphi : \partial \dd \to \cc$ as follows. Fixing $\tau \in (0,1]$, in the set $\{z \in \partial \dd \colon \re(z) \geq 0\}$ we let $\varphi(z) = \Phi_\tau(z)$. We also map the complementary set $\{z \in \partial \dd \colon \re(z) < 0\}$ smoothly into the complement of $\bar{\yy_\tau}$, and in such a way that $\varphi(\partial \dd)$ becomes the boundary of a Jordan domain $\tilde{\yy}$. See Figure \ref{phiFig} for an illustration.
\begin{figure}[h]
\includegraphics[scale=0.3]{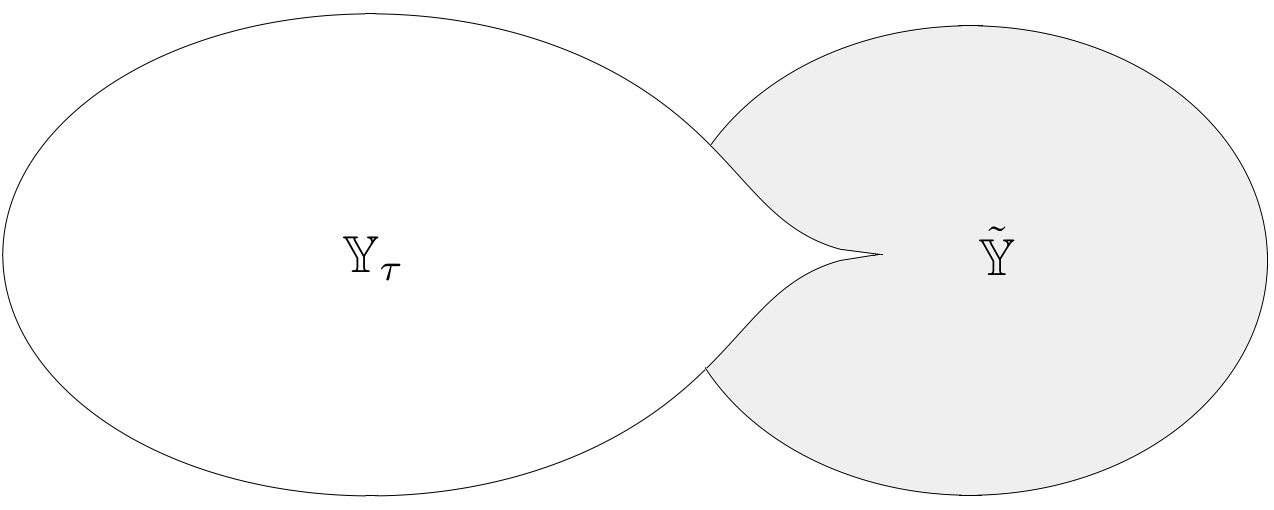}
\caption{The Jordan domains $\yy_\tau$ and $\tilde{\yy}$.}
\label{phiFig}
\end{figure}
It is now easy to see that the map $\varphi$ satisfies the Douglas condition \eqref{eq:douglas}. Indeed, since the map $\Phi_\tau$ is in the Sobolev space $\W^{1,2}(\dd, \C)$ its restriction to the boundary must necessarily satisfy the Douglas condition. Since the map $\varphi$ aligns with this boundary map in a neighborhood of the point $z=1$, verifying the finiteness of the integral in \eqref{eq:douglas} poses no difficulty in this neighborhood. On the rest of the boundary of $\partial\dd$ we may choose $\varphi$ to be locally Lipschitz, which shows that \eqref{eq:douglas} is necessarily satisfied for $\varphi$. Hence we have found a map from $\partial \dd$ into the boundary of the chord-arc domain $\tilde{\yy}$ which admits a $\W^{1,2}$-extension to $\dd$ but not a homeomorphic one.\\\\
\textbf{Example \ref{ex:yi}.} In \cite{Zh}, Zhang constructed an example of a Jordan domain, which we shall denote by $\yy$, so that the conformal map $g : \dd \to \yy$ does not admit a $\W^{1,1}$-homeomorphic extension to the whole plane. We shall not repeat this construction here, but will instead briefly show how it relates to our questions.

The domain $\yy$ is constructed in such a way that there is a boundary arc $\Gamma \subset \partial \yy$ over which one cannot extend the conformal map $g$ even locally as a $\W^{1,1}$-homeomorphism. The complementary part of the boundary $\yy \setminus \Gamma$ is piecewise linear. Hence we may employ the same argument as in the previous example. We choose a Jordan domain $\tilde{\yy}$ in the complement of $\yy$ whose boundary consists of the arc $\Gamma$ and, say, a piecewise linear curve. We then define a boundary map $\varphi : \partial \dd \to \partial \tilde{\yy}$ so that it agrees with $g$ in a neighborhood of the set $g^{-1} (\Gamma)$ and is locally Lipschitz everywhere else. With the same argument as before, this boundary map must satisfy the Douglas condition \eqref{eq:douglas}. Hence this boundary map admits a $\W^{1,2}$-extension to $\dd$ but not even a $\W^{1,1}$-homeomorphic extension. Naturally the boundary of the domain $\tilde{\yy}$ is quite ill-behaved, in particular nonrectifiable (though the Hausdorff dimension is still one).

\begin{proof}[Proof of Theorem~\ref{thm:dirichlet}]
Let $\gamma : \partial \dd \to \partial \yy$ denote a constant speed parametrization of the rectifiable curve $\partial \yy$. Let $G: \cc \to \cc$ be the homeomorphic Lipschitz extension of $\gamma$ given by Theorem \ref{thm:lip}. Denoting $f := \varphi^{-1} \circ \gamma$, we find by change of variables that
\begin{align*}\int_{\partial \mathbb Y} \int_{\partial \mathbb Y} \left|\log \abs{\varphi^{-1} (\xi) - \varphi^{-1} (\eta) }  \right| \abs{\dtext \xi } \abs{\dtext \eta }  =
\int_{\partial \dd} \int_{\partial \dd} \left|\log \abs{f (z) - f (\omega) }  \right| \abs{\dtext z }  \abs{\dtext \omega }  .
\end{align*}
Now the result of Astala, Iwaniec, Martin and Onninen~\cite{AIMO} shows that the inverse map $f^{-1} : \partial \xx \to \partial \dd$ satisfies the Douglas condition \eqref{eq:douglas}. Thus $f^{-1}$ extends to a harmonic $\W^{1,2}$-homeomorphism $H_1$ to $\bar{\dd}$ by the RKC-Theorem (Lemma~\ref{lem:RKC}). Letting $h := G \circ H_1$, we find that $h$ lies in the space $\W^{1,2}(\xx)$ since $G$ is Lipschitz. Moreover, the boundary values of $h$ are equal to $\gamma \circ (\varphi^{-1} \circ \gamma)^{-1} = \varphi$, giving us a homeomorphic extension of $\varphi$ in the Sobolev space $\W^{1,2}(\xx)$.

To further extend $\varphi$ into the complement of $\xx$, assume first without loss of generality that $0 \in \xx$ and $0 \in \yy$. We now let $\tau(z) = 1/\bar{z}$ denote the inversion map, which is a diffeomorphism in $\cc \setminus \{0\}$. The map $\psi := \tau \circ \varphi \circ \tau$ is then a homeomorphism from $\partial \tau(\xx)$ to $\partial \tau(\yy)$, and must also satisfy the condition \eqref{eq:invdouglas} due to the bounds on $\tau$. The earlier part of the proof shows that we may extend $\psi$ as a $\W^{1,2}$-homeomorphism $\tilde{h}$ from the Jordan domain bounded by  $\partial \tau(\xx)$ to the Jordan domain bounded by $\partial \tau(\yy)$. Hence the map $\tau \circ \tilde{h} \circ \tau$ is a $\W^{1,2}_{loc}$-homeomorphism from the complement of $\xx$ to the complement of $\yy$ and equal to $\varphi$ on the boundary. This concludes the proof.

\end{proof}

\section{The multiply connected case, Proof of Theorem~\ref{thm:multiply}}\label{anyplansguysz}

In this section we   consider multiply connected Jordan domains $\X$ and $\Y$ of the same topological type. Any such domains can be equivalently obtained by removing from a simply connected Jordan domain  the same number, say $0\le k < \infty$, of closed disjoint topological disks. If $k=1$, the obtained doubly connected domain is conformally equivalent with a circular annulus $\mathbb A = \{z \in \mathbb C \colon  r < \abs{z} <1\}$ with some $0<r<1$. In fact, if $k \ge1$ every $(k+1)$-connected Jordan domain can be mapped by a conformal mapping  onto a {\it circular domain}, see~\cite{Gob}. In particular we may consider a $(k+1)$-connected circuilar domain consisting of the domain bounded by the boundary of the unit disk $\mathbb D $ and $k$ other circles  (including points) in the interior of $\mathbb D$. The conformal mappings between multiply connected Jordan domains extends continuously up to the boundaries. 

The idea of the proof of  Theorem~\ref{thm:multiply} is simply to split the multiply connected domains $\xx$ and $\yy$ into simply connected parts and apply Theorem \ref{thm:main} in each of these parts. Let us consider first the case where $\xx$ and $\yy$ are doubly connected.
\subsection{Doubly connected $\X$ and $\Y$} $\,$ 

\emph{Case 1.} $p=1$. Suppose that the boundary of $\xx$ is rectifiable. We split the domain $\xx$ into two rectifiable simply connected domains as follows. Take a line $\ell$ passing through any point in the bounded component of $\cc \setminus \xx$. Then necessarily there exist two open line segments $I_1$ and $I_2$ on $\ell$ such that these segments are contained in $\xx$ and their endpoints lie on different components of the boundary of $\xx$. These segments split the domain $\xx$ into two rectifiable Jordan domains $\xx_1$ and $\xx_2$. 

For $k=1,2$, let $p_k$ denote the endpoint of $I_k$ lying on the inner boundary of $\xx$ and $P_k$ the endpoint on the outer boundary. We let $q_k = \varphi(p_k)$ and $Q_k = \varphi(P_k)$. We would now simply like to connect $q_k$ with $Q_k$ by a rectifiable curve $\gamma_k$ inside of $\yy$ such that $\gamma_1$ and $\gamma_2$ do not intersect. It is quite obvious this can be done but we provide a proof regardless.

Let $\yy_+$ denote the Jordan domain bounded by the outer boundary of $\yy$. Take a conformal map $g_+ : \dd \to \yy_+$. Then $g_+'$ is in the Hardy space $H^1$ since $\partial \yy_1$ is rectifiable, and we find by Theorem 3.13 in~\cite{Duhp} that $g_+$ maps the segment $[0,g_+^{-1}(Q_k)]$ into a rectifiable curve in $\yy_+$. Let $\gamma_k^+$ denote the image of the segment $[(1-\epsilon)g_+^{-1}(Q_k),g_+^{-1}(Q_k)]$ under $g_+$ for a sufficiently small $\epsilon$. Hence we have a rectifiable curve $\gamma_k^+$ connecting $Q_k$ to an interior point $Q_k^+$ of $\yy$ if $\epsilon$ is small enough. With a similar argument, possibly adding a M\"obius transformation to the argument to invert the order of the boundaries, one finds a rectifiable curve $\gamma_k^-$ connecting $q_k$ to an interior point $q_k^-$. For small enough $\epsilon$ the four curves constructed here do not intersect.

If $\Gamma$ denotes the union of these four curves, we may now use the path-connectivity of the domain $\yy \setminus \Gamma$ to join the points $Q_1^+$ and $q_1^-$ with a smooth simple curve inside $\yy$ that does not intersect $\Gamma$. By adding the curves $\gamma_1^+$ and $\gamma_1^-$ one obtains a rectifiable simple curve $\gamma_1$ connecting $Q_1$ and $q_1$. Using the fact that $\yy \setminus \Gamma$ is doubly connected, we may now join $Q_2^+$ and $q_2^-$ with a smooth curve that does not intersect $\gamma_1$ nor $\Gamma$. This yields a rectifiable simple curve $\gamma_2$ connecting $Q_2$ and $q_2$. This proves the existence of the curves $\gamma_k$ with the desired properties. These curves split $\yy$ into two simply connected Jordan domains $\yy_1$ and $\yy_2$.

We may now extend the homeomorphism $\varphi$ to map the boundary of $\xx_k$ to the boundary of $\yy_k$ homeomorphically. The exact parametrization which maps the segments $I_k$ to the curves $\gamma_k$ does not matter. The rest of the claim follows directly from the first part of Theorem \ref{thm:main}, giving us a homeomorphic extension of $\varphi$ in the Sobolev class $\W^{1,1}(\xx, \mathbb C)$, as claimed.
\\\\
\emph{Case 2.} $1<p<2$. Suppose that $\xx$ has $s$-hyperbolic growth. Then we take an annulus $\A$ centered at the origin such that there exists a conformal map $g : \A \to \xx$. By a result of Gehring and Osgood~\cite{GO}, the quasihyperbolic metrics $h_{\xx}$ and $h_{\A}$ are comparable via the conformal map $g$. This shows that for any fixed $x_0 \in \A$ and all $x \in \A$ we have
\begin{equation}\label{hyperEq1}h_{\A}(x_0,x) \leq C h_{\xx}(g(x_0),g(x)) \leq \frac{C}{\dist(g(x),\partial\xx)^{1-s}}.\end{equation}
Let now $\A_+$ denote the simply connected domain obtained by intersecting $\A$ and the upper half plane. We claim that the the domain $\xx_+ := g(\A_+)$ has $s$-hyperbolic growth as well.

To prove this claim, fix $x_0 \in \A_+$ and take arbitrary $x \in \A$. Let $d = \dist(x,\partial \A_+)$. We aim to establish the inequality
\begin{equation}\label{hyperEq2}h_{\A_+}(x_0,x) \leq \frac{C}{\dist(g(x),\partial\xx_+)^{1-s}}.\end{equation}
Note that $\A_+$ is bi-Lipschitz equivalent with the unit disk, implying that $h_{\A_+}(x_0,x)$ is comparable to $\log(1/d)$. Since the boundary of $\A_+$ contains two line segments on the real line, let us denote them by $I_1$ and $I_2$. Note that we have the estimate
\begin{equation}\label{distgEstim1}\dist(g(x),\partial \xx_+) \leq \dist(g(x),\partial \xx).\end{equation}
If it would happen that $d = \dist(x,\partial \A)$, meaning that the closest point to $x$ on $\partial \A_+$ is not on $I_1$ or $I_2$, then the hyperbolic distances $h_{\A_+}(x_0,x)$ and $h_{\A}(x_0,x)$ are comparable and by the inequalities \eqref{hyperEq1} and \eqref{distgEstim1} the inequality \eqref{hyperEq2} holds. It is hence enough to prove \eqref{hyperEq2} in the case when $d = \dist(x,I_1 \cup I_2)$. We may also assume that $d$ is small. Due to the geometry of the  half-annulus $\A_+$, the vertical line segment $L_x$ between $x$ and its projection to the real line lies on either $I_1$ or $I_2$ and its length is $d$. Letting $D$ denote the distance of $x$ to $\partial\A_+ \setminus (I_1 \cup I_2)$, we have that $D \geq d$.

We may now reiterate the proof of \eqref{hyperbolicEstim3}  to find that
\[|g'(z)| \leq \frac{C}{\dist(z,\partial \A) \log^{\frac{1}{1-s}}(\dist(z,\partial \A)^{-1})}\]
for $z \in \A$. We should mention that the simply connectedness assumption used in the proof of  \eqref{hyperbolicEstim3} may be circumvented by using the equivalence of the quasihyperbolic metrics under $g$ instead of passing to the hyperbolic metric. Hence
\[\dist(g(x),\partial \xx_+) \leq \int_{L_x} |g'(z)| |dz| \leq \frac{C d}{D \log^{\frac{1}{1-s}}(1/D)}.\]
From this we find that \eqref{hyperEq2} is equivalent to
\[\log(1/d) \leq C\frac{D^{1-s} \log(1/D)}{d^{1-s}},\]
which is true since $D \geq d$. Hence \eqref{hyperEq2} holds, and this implies that $\xx_+$ has $s$-hyperbolic growth by reversing the argument that gives \eqref{hyperEq1}.

We define $\xx_-$ similarly. Hence we have split $\xx$ into two simply connected domains with $s$-hyperbolic growth. On the image side, we may split $\yy$ into two simply connected domains with rectifiable boundary as in Case 1. Extending $\varphi$ in an arbitrary homeomorphic way between the boundaries of these domains and applying part 2 of Theorem \ref{thm:main} gives a homeomorphic extension of $\varphi$ in the Sobolev class $\W^{1,p}(\xx, \mathbb C)$ whenever $s > p-1$.
\subsection{The general case} $\,$

\emph{Case 3.}  $p=1$. Assume that $\xx$ and $\yy$ are $\ell$-connected Jordan domains with rectifiable boundaries. By induction, we may assume that the result of Theorem \ref{thm:multiply} holds for $(\ell-1)$-connected Jordan domains. Hence we are only required to split $\xx$ and $\yy$ into two domains with rectifiable boundary, one which is doubly connected and another which is $(\ell-1)$-connected.

We hence describe how to 'isolate' a given boundary component $X_0$ from a $\ell$-connected Jordan domain $\xx$. Let $X_{outer}$ denote the outer boundary component of $\xx$. Take a small neighborhood of $X_0$ inside $\xx$. Let $\gamma_0$ be a piecewise linear Jordan curve contained in this neighborhood and separating $X_0$ from the rest of the boundary components of $\xx$. Let also $\gamma_1$ be a piecewise linear Jordan curve inside $\xx$ and in a small enough neighborhood of $X_{outer}$ so that all of the other boundary components of $\xx$ are contained inside $\gamma_1$. Take $y_0$ and $y_1$ on $\gamma_0$ and $\gamma_1$ respectively, and connect them with a piecewise linear curve $\alpha_y$ not intersecting any boundary components of $\xx$. Choose $z_0$ and $z_1$ close to $y_0$ and $y_1$ respectively so that we may connect $z_0$ and $z_1$ by a piecewise linear curve $\alpha_z$ arbitrarily close to $\alpha_y$ but neither intersecting it nor  any boundary components of $\xx$. Since the region bounded by $X_{outer}$ and $\gamma_1$ is doubly connected, by the construction in Case 1 we may connect $y_1$ and $z_1$ with any two given points $y_2$ and $z_2$ on the boundary $X_{outer}$ via non-intersecting rectifiable curves $\beta_y$ and $\beta_z$ lying inside this region.

Let now $\Gamma$ denote the union of the curves $\beta_y$, $\beta_z$, $\alpha_y$, $\alpha_z$, and the curve $\gamma_0'$ obtained by taking the curve $\gamma_0$ and removing the part between $y_0$ and $z_0$. By construction $\Gamma$ contains two arbitrary points on $\xx_{outer}$ and separates the domain $\xx$ into a doubly connected domain with inner boundary component $X_0$ and a $(n-1)$-connected Jordan domain. Since $\Gamma$ is rectifiable, both of these domains are also rectifiable.

Applying the same construction for $\yy$, we may separate the boundary component $\varphi(X_0)$ of $\yy$ by a rectifiable curve $\Gamma'$. Since the boundary points $y_2$ and $z_2$ above were arbitrary, we may assume that $\Gamma'$ intersects the outer boundary of $\yy$ at the points $\varphi(y_2)$ and $\varphi(z_2)$. Extending $\varphi$ to a homeomorphism from $\Gamma$ onto $\Gamma'$ and applying the induction assumptions now gives a homeomorphic extension in the class $\W^{1,1}(\xx , \mathbb C)$.
\\\\
\emph{Case 4.} $1<p<2$. We still have to deal with the case where $\xx$ has $s$-hyperbolic growth and is $\ell$-connected. By the same arguments as in the previous case, it will be enough to split $\xx$ into a doubly connected and $(\ell-1)$-connected domain with $s$-hyperbolic growth.

Since $\xx$ is $\ell$-connected, there exists a domain $\Omega$ such that every boundary component of $\Omega$ is a circle and there is a conformal map $g : \Omega \to \xx$. Let $\Gamma \subset \Omega$ be a piecewise linear curve separating one of the inner boundary components of $\partial\Omega$. Hence $\Omega$ splits into a doubly connected set $\Omega_1$ and a $(\ell-1)$-connected set $\Omega_2$. We claim that the domains $\xx_1 = g(\Omega_1)$ and $\xx_2 = g(\Omega_2)$ have $s$-hyperbolic growth.

The proof of this claim is nearly identical to the arguments in Case 2, so we will summarize it briefly. For $\xx_2$, we aim to establish the inequality
\begin{equation}\label{hyperEq3}
h_{\Omega_2}(x_0,x) \leq \frac{C}{\dist(g(x),\partial\xx_2)^{1-s}}
\end{equation}
for fixed $x_0 \in \Omega_2$ and $x \in \Omega_2$. For this inequality, it is only essential to consider $x$ close to $\partial \Omega_2$. If $x$ is closer to the boundary of the original set $\partial \Omega$ than to $\Gamma$, then the hyperbolic distance of $x_0$ and $x$ in $\Omega_2$ is comparable to the distance inside the larger set $\Omega$. Then the $s$-hyperbolic growth of $\Omega$ implies \eqref{hyperEq3} as in Case 2. If $x$ is closer to $\Gamma$ but a fixed distance away from the boundary of $\Omega$, then the smoothness of $g$ in compact subsets of $\Omega$ implies the result. If $x$ is closest to a line segment in $\Gamma$ which has its other endpoint on $\partial \Omega$, then we may employ a similar estimate as in Case 2, using the bound for $|g'(z)|$ in terms of $\dist(z,\partial \Omega)$, to conclude that \eqref{hyperEq3} also holds here. This implies that $\xx_2$ satisfies \eqref{hyperEq3}, and hence it has $s$-hyperbolic growth. The argument for $\xx_1$ is the same.

After splitting $\xx$ into two domains of smaller connectivity and $s$-hyperbolic growth, we split the target $\yy$ accordingly into rectifiable parts using the argument from Case 3. Applying induction on $n$ now proves the result. This finishes the proof of Theorem \ref{thm:multiply}.

\section{Monotone Sobolev minimizers}\label{sec:mono}
The classical {\it  harmonic mapping problem}  deals with the question of whether there exists a harmonic homeomorphism between two given domains. Of course, when the domains are Jordan such a mapping problem is always solvable. Indeed, according to the Riemann Mapping Theorem there is a conformal mapping $h \colon \overline{\X} \onto \overline{\Y}$. Finding a harmonic  homeomorphism which coincides with  a given boundary homeomorphism $\varphi \colon \partial \X \onto \partial \Y$ is a more subtle question.  If $\Y$ is convex, then there always exists a harmonic homeomorphism $h \colon \overline {\X} \onto \overline{\Y}$ with $h (x)= \varphi (x)$ on $\partial \X$ by Lemma~\ref{lem:RKC}. For a non-convex target $\Y$, however,  there always exists at least one boundary homeomorphism whose harmonic extension takes points in $\X$ beyond $\overline{\Y}$. To find a deformation $h \colon \overline{\X} \onto \overline{\Y}$ which resembles harmonic homeomorphisms Iwaniec and Onninen~\cite{IOmhh} applied the direct method in the calculus of variations and considered minimizing sequences  in $\mathscr H_\varphi^{1,2} (\overline{\X} , \overline{\Y})$. They called such minimizers {\it monotone Hopf-harmonics} and proved the existence and uniqueness result in  the case when $\Y$ is a Lipschitz domain and the boundary data $\varphi$ satisfies the Douglas condition.  Note that by the Riemann Mapping Theorem one may  always assume that $\X= \mathbb D$. Theorem~\ref{thm:dirichlet} opens up such studies beyond the Lipschitz targets.  
Indeed, under the assumptions of Theorem~\ref{thm:dirichlet}, the class  $\mathscr H_\varphi^{1,2} (\overline{\mathbb D} , \overline{\Y})$ is non-empty. Furthermore, if $h_\circ \in \mathscr H_\varphi^{1,2} (\overline{\mathbb D} , \overline{\Y})$, then $h_\circ$ satisfies the uniform modulus of continuity estimate
\[\abs{h_\circ (x_1) - h_\circ (x_2)}^2 \le C \frac{\int_{\mathbb D} \abs{Dh_\circ}^2}{\log \left( \frac{1}{\abs{x_1-x_2}}\right) }\]
for $x_1, x_2 \in \mathbb D$ such that $\abs{x_1-x_2} <1$. This follows from taking the global $\W^{1,2}_{loc}$-homeomorphic extension given by Theorem~\ref{thm:dirichlet} and applying a standard local modulus of continuity estimate for $\W^{1,2}$-homeomorphisms, see~\cite[Corollary 7.5.1 p.155]{IMb}. Now, applying the direct method in the calculus of variations allows us to find a minimizing sequence in $\mathscr H_\varphi^{1,2} (\overline{\mathbb D} , \overline{\Y})$  for the Dirichlet energy which converges weakly in $\W^{1,2} (\mathbb D , \mathbb C)$ and uniformly in $\overline{ \mathbb D}$. Being a  uniform limit of homeomorphisms the limit mapping $H \colon \overline{\mathbb D} \onto \overline{\Y} $  becomes {\it monotone}.  Indeed,  the classical Youngs approximation theorem~\cite{Yo} which  asserts that   a  continuous map between compact oriented topological 2-manifolds (surfaces) is monotone  if and only if it is a uniform limit of homeomorphisms. {Monotonicity, the concept of Morrey~\cite{Mor}, simply means that for a continuous $H \colon \overline{\X} \to \overline{\Y}$ the preimage $H^{-1} (y_\circ)$ of a point $y_\circ \in \overline{\Y}$ is a continuum in $\overline{\X}$. We have hence just given a proof of the following result.

\begin{theorem}
Let $\X$ and $\Y$ be Jordan domains and assume that $\partial \Y$ is rectifiable. If $\varphi \colon \partial \X \onto \partial \Y$ satisfies~\eqref{eq:invdouglas}, then there exists a monotone Sobolev mapping $H \colon \overline{\X} \onto \overline{\Y}$ in $\W^{1,2} (\X, \mathbb C)$ such that $H$ coincides with $\varphi$ on $\partial \X$ and
\[ \int_{\X} \abs{DH(x)}^2 \, \dtext x = \inf_{h \in \mathscr H_\varphi^{1,2} (\overline{\X} , \overline{\Y})} \int_\X \abs{Dh(x)}^2 \, \dtext x \, .  \]
\end{theorem}

\end{document}